\documentclass[11pt]{amsart}

\usepackage[a4paper]{geometry}
\usepackage{latexsym, amsthm, amsfonts, amsmath,amssymb,graphicx,lmodern,mathrsfs}

\usepackage[latin1]{inputenc} 
\usepackage[T1]{fontenc} 
\usepackage[english]{babel} 

\newtheorem{theo}{Theorem}[section]
\newtheorem{lem}[theo]{Lemma}
\newtheorem{propo}[theo]{Proposition}
\newtheorem{coro}[theo]{Corollary}

\theoremstyle{definition}
\newtheorem{defi}[theo]{Definition}

\theoremstyle{remark}
\newtheorem{rem}[theo]{Remark}
\newtheorem{ex}[theo]{Example}

\makeatother

\def\Z{\mathbb{Z}}
\def\C{\mathbb{C}}
\def\N{\mathbb{N}}
\def\Q{\mathbb{Q}}

\def\z{\zeta}
\def\n{\eta}

\def\s{\sigma}
\def\a{\alpha}
\def\e{\varepsilon}

\def\b{\beta}

\def\n'{\nu}

\def\l{\lambda}

\def\k{\kappa}

\def\T{\Theta}

\def\sq {\sigma_{q}}

\begin{document}
\sloppy
\title[$q$-Borel-Laplace summation]{$q$-Borel-Laplace summation for $q$-difference equations with two slopes.}
\author{Thomas Dreyfus}
\address{Univ Lyon, Université Claude Bernard Lyon 1, CNRS UMR 5208, Institut Camille Jordan, 43 blvd. du 11 novembre 1918, F-69622 Villeurbanne cedex, France.}
\email{dreyfus@math.univ-lyon1.fr}
\author{Anton Eloy}
\address{Université Paul Sabatier - Institut de Mathématiques de Toulouse,}
\curraddr{118 route de Narbonne, 31062 Toulouse}
\email{Anton.Eloy@math.univ-toulouse.fr}
\thanks{The first author is funded by the labex CIMI. This project has received funding from the European Research Council (ERC) under the European Union's Horizon 2020 research and innovation programme under the Grant Agreement No 648132. The final version of this paper will be published in Journal of Difference Equations and Applications.}



\date{\today}


\begin{abstract} 
In this paper, we consider linear $q$-difference systems with coefficients that are germs of meromorphic functions, with Newton polygon that has two slopes. Then, we explain how to compute similar meromorphic gauge transformations than those appearing in the work of Bugeaud, using a $q$-analogue of the Borel-Laplace summation. 
\end{abstract} 

\maketitle

\tableofcontents
\pagebreak[3]
\section*{Introduction}
Let $q$ be a complex number with $|q|>1$, and let us define $\sq$, the dilatation operator that sends $y(z)$ to $y(qz)$. Divergent formal power series may appear as solutions of  linear $q$-difference systems with coefficients in $\C(z)$.
 The simplest example is the $q$-Euler equation
$$z\sq y+y=z,$$
that admits as solution the formal power series with complex coefficients $$\hat{h}:=\displaystyle\sum_{\ell\in \N}(-1)^{\ell}q^{\frac{\ell(\ell+1)}{2}}z^{\ell+1}.$$
To construct a meromorphic solution of the above equation, we use the Theta function 
$$\T_{q} (z):=\displaystyle \sum_{\ell \in \Z} q^{\frac{-\ell(\ell+1)}{2}}z^{\ell}=\displaystyle \prod_{\ell \in \N} (1-q^{-\ell-1})(1+q^{-\ell-1}z)(1+q^{-\ell}z^{-1}),$$ 
 which is analytic on $\C^{*}$, vanishes on the discrete $q$-spiral $-q^{\Z}:=\{-q^{n},n\in \Z\}$, with simple zero, and satisfies $$\s_{q}\T_{q}(z)=z\T_{q}(z)=\T_{q}\left(z^{-1}\right).$$
For $\l\in\C^*$ we denote by $[\l]$ the class of $\l$ in $\C^*/q^\Z$. 
Then, for all $[\l]\in \left(\C^{*}/q^{\Z}\right)\setminus \{[-1]\}$ the following function $\hat{h}^{[\l]}$, is solution of the same equation as $\hat{h}$ and is meromorphic on $\C^{*}$ with simple poles on the $q$-spiral $-\l q^{\Z}$:
$$\hat{h}^{[\l]}:=\displaystyle \sum_{\ell\in \Z}\frac{q^{\ell}\l}{1+q^{\ell}\l}\times\frac{1}{\T_{q}\left(\frac{q^{1+\ell}\l}{z}\right)}.$$
It is easy to see that it does not depend of the chosen representative $\l$ of the class $[\lambda]$.\\
More generally, the meromorphic solutions play a major role in Galois theory of $q$-difference equations, see \cite{BuT,DVRSZ,RS07,RS09,RSZ,S00}. In the usual Picard-Vessiot theory, the Galois group of a $q$-difference equation with coefficients in $\C(z)$ is defined over $\C$, but the solutions in the Picard-Vessiot rings involve symbols. On the other hand, Praagman in \cite{Pra} has shown the existence of a fundamental solution for such system, that is an invertible solution matrix, with entries that are meromorphic on $\C^{*}$. Applying naively  the classical Picard-Vessiot theory with those solutions would give a Galois group defined on the field of functions invariant under the action of~$\sq$, that is, the field of meromorphic functions on the torus $\C^{*}/ q^{\Z}$, rather than a group defined on $\C$. Fortunately, Sauloy has explained how to recover the Galois group with the meromorphic solutions.

 Let us consider 
\begin{equation}\label{eq1}
\sq Y=AY, \hbox{ where } A\in \mathrm{GL}_{m}(\C(\{z\})),
\end{equation}
which means that $A$ is an invertible $m\times m$ matrix with coefficients that are germs of meromorphic functions at $z=0$. We are going to assume that the slopes of the Newton polygon of (\ref{eq1}) belong to~$\Z$. See \cite{RSZ}, $\S 2.2$, for a precise definition. Then, the work of Birkhoff and Guenther implies that after an analytic gauge transformation, such system may be put into a very simple form, that is in the Birkhoff-Guenther normal form. Moreover, after a formal gauge transformation, a system in the Birkhoff-Guenther normal form may be put into a diagonal bloc system. Then, the authors of \cite{RSZ} build a set of meromorphic gauge transformations, that make the same transformation as the formal gauge transformation, with germs of entries, having poles contained in a finite number of~$q$-spirals of the form $q^{\Z}\a$, with $\a\in \C^{*}$. Moreover, the germs of meromorphic gauge transformations they build are uniquely determined by the set of poles and their multiplicities.\par 
V. Bugeaud considers the case where the Newton polygon of (\ref{eq1}) has two slopes that are not necessarily entire. Analogously to \cite{RSZ} she builds meromorphic gauge transformations, but this time, the germs of meromorphic entries have poles on $q^{d\Z}$-spirals, for some $d\in \N^{*}$, instead of $q^{\Z}$-spirals. See \cite{BuT}.\\\par 
In the differential case, we also have the existence of formal power series solutions of linear differential equations with coefficients in $\C(\{z\})$, but we may apply to them a Borel-Laplace summation, in order to obtain meromorphic solutions. See \cite{B,R85,M95,VdPS}. In the $q$-difference case, although there are several $q$-analogues of the Borel and Laplace transformations, see \cite{Abdi60,Abdi64,D4,D3,DVZ,Hahn,MZ,R92,RZ,Taha,Z99,Z00,Z01,Z02,Z03}, we do not know how to compute in general the meromorphic gauge transformations of \cite{BuT} using $q$-analogues of the Borel and Laplace transformations. Remark that when~(\ref{eq1}) has two entire slopes, we are in the framework of \cite{RSZ}, and \cite{D4}, Theorem~2.4, explains how to compute them with $q$-analogues of the Borel and Laplace transformations. See Remark~\ref{rem1}. Although the $q$-summation process of \cite{D4} applies for $q$-difference equations with two slopes in general, the functions obtained have too many poles to be optimal. Indeed, if the slopes are $0$ and $\frac{n}{d}$ with $n,d$ coprimes  numbers, we expect to find meromorphic gauge transformations having entries with $n$ $q^{d\Z}$ spirals of poles of order at most one, and those of \cite{D4} have  $n$ $q^{\Z}$ spirals of poles of order at most one.
 The main goal of this paper is to compute Bugeaud's like meromorphic gauge transformations with a $q$-Borel-Laplace summation and with a minimal number of poles. \\ \par 
The paper is organized as follows. In $\S \ref{sec1}$ we introduce the $q$-analogues of the Borel and Laplace transformations. In $\S \ref{sec2}$, we give a brief summary of \cite{BuT}. The meromorphic gauge transformations she builds have $q^{d\Z}$-spirals of poles of multiplicity at most $n$, for some $n\in \N^{*}$. We prove the existence and the uniqueness of meromorphic gauge transformations having $q^{\frac{d\Z}{n}}$-spiral of poles of multiplicity at most $1$. In $\S \ref{sec3}$, we explain how to compute the latter meromorphic gauge transformations using a $q$-analogue of the Borel-Laplace summation. \\ \par 

\textbf{Acknowledgments.} The authors would like the thank the anonymous referee for permitting them to correct a mistake that was originally made in the first version of the paper.
\pagebreak[3]
\section{Definition of $q$-analogues of Borel and Laplace transformations.}\label{sec1}
We start with the definition of the $q$-Borel and the $q$-Laplace transformations. The definition of the $q$-Borel we use comes from  \cite{D4}. Note that when $\mu=1$, we recover the $q$-Borel transformation that was originally introduced in \cite{R92}. The $q$-Laplace we use is slightly different than the transformation defined in \cite{D4}. When $q>1$ is real and $q$ is replaced by $q^{1/\mu}$, we recover the $q$-Laplace transformation and the functional space introduced in \cite{Z02}, Theorem 1.2.1. Earlier introductions of $q$-Laplace transformations can be found in \cite{Abdi60,Abdi64}. Those latter behave differently than our $q$-Laplace transformation, since they involve $q$-deformations of the exponential, instead of the Theta function. \\ \par 

Throughout the paper, we will say that two analytic functions $f$ and $g$ are equal, if there exists $U\subset \C$, open connected set, such that $f$ and $g$ are analytic on $U$ and are equal on $U$. We fix, once for all, a determination of the logarithm over the Riemann surface of the logarithm we call $\log$. If $\a\in \C$, and $b\in \C^{*}$, then, we write $b^{\a}$ instead of~$e^{\a\log (b)}$. One has $b^{\a+\b}=b^{\a}b^{\b}$, for all $\a,\b\in \C$, and all $b\in \C^{*}$  Let $\C[[z]]$ be the ring of formal power series.

\pagebreak[3]
\begin{defi}\label{defi1}
Let $\mu\in \Q_{>0}$. We define the $q$-Borel transformation of order $\mu$ as follows
$$
\begin{array}{llll}
\hat{\mathcal{B}}_{\mu}:&\C[[z]]&\longrightarrow&\C[[\z]]\\
&\displaystyle\sum_{\ell\in \N} a_{\ell}z^{\ell }&\longmapsto&\displaystyle\sum_{\ell\in \N} a_{\ell}q^{\frac{-\ell(\ell-1)}{2\mu }} \z^{\ell }.
\end{array}
$$
\end{defi}

We are going to make the following abuse of notations. For $\l\in\C^*$, and $n,d$ positive co-prime integers, we denote by $[\l]$ the class of $\l$ in $\C^*/q^{\frac{d\Z}{n}}$. Let $\mathcal{M}(\C^{*},0)$ be the field of functions that are meromorphic on some punctured neighborhood of $0$ in~$\C^{*}$. More explicitly, $f\in \mathcal{M}(\C^{*},0)$ if and only if there exists $V$, an open neighborhood of $0$, such that $f$ is analytic on $V\setminus \{0\}$.
\begin{defi}
Let $\mu=\frac{n}{d}\in \Q_{>0}$ with $n,d$ positive co-prime integers and $[\l] \in \C^{*}/q^{\frac{d\Z}{n}}$.   An element $f$ of $\mathcal{M}(\C^{*},0)$ is said to belongs to $\mathbb{H}_{\mu}^{[\l]}$ if there exist~$\e>0$ and a connected domain $\Omega \subset \C$, such that:
\begin{itemize}
\item  $\displaystyle\bigcup_{\ell\in \frac{d\Z}{n}}\Big\{z\in \C^{*}\Big| \left|z-\l q^{\ell}\right|<\e\left| q^{\ell}\l\right| \Big\}\subset \Omega.$
\item The function $f$ can be continued to an analytic function on $\Omega$ with $q^{\frac{d}{n}}$-exponential growth at infinity, which means that there exist constants $L,M>0$, such that for all $z\in \Omega$:
$$|f(z)|<L\T_{|q|^{\frac{d}{n}}}(M|z|) .$$
\end{itemize}
Note that it does not depend of the chosen representative $\l$.
An element $f$ of $\mathcal{M}(\C^{*},0)$ is said to belongs to $\mathbb{H}_{\mu}$, if there exists a finite set~${\Sigma \subset \C^{*}/q^{\frac{d\Z}{n}}}$, such that for all ${[\l] \in \left(\C^{*}/q^{\frac{d\Z}{n}}\right)\setminus \Sigma }$, we have $f\in \mathbb{H}_{\mu}^{[\l]}$.
\end{defi}

\begin{defi}
Let $\mu=\frac{n}{d}\in \Q_{>0}$ with $n,d$ positive co-prime integers, $[\l] \in \C^{*}/q^{\frac{d\Z}{n}}$. Because of \cite{D4}, Lemma 1.3, the following map, which does not depend of the chosen representative $\l$, is well defined and is called $q$-Laplace transformation of order $\mu$: 

$$
\begin{array}{llll}
\mathcal{L}_{\mu}^{[\l]}:&\mathbb{H}_{\mu}^{[\l]}&\longrightarrow&\mathcal{M}(\C^{*},0)\\
&f&\longmapsto&
\displaystyle \sum_{\ell\in \frac{d\Z}{n}}\frac{f\left(q^{\ell}\l\right)}{\T_{q^{\frac{d}{n}}}\left(\frac{q^{\frac{d}{n}+\ell}\l}{z}\right)}.
\end{array}
$$

For $|z|$ close to $0$, the function $\mathcal{L}_{\mu}^{[\l]}\left(f\right)(z)$ has poles of order at most $1$ that are contained in the~$q^{\frac{d\Z}{n}}$-spiral $-\l q^{\frac{d\Z}{n}}$. 
\end{defi}

\begin{rem}\label{rem3}
The $q$-Laplace we use is slightly different than the transformation defined in \cite{D4}. Let $\mu\in \Q_{>0}$ and $q':=q^{1/\mu}$. The $q'$-Laplace transformation of order $(1,1)$ of \cite{D4} equals to the $q$-Laplace transformation of order $\mu$ of this paper.
\end{rem}

The following lemmas, which give basic properties of the Borel and Laplace transformations, will be needed in $\S \ref{sec3}$. For $n,d$ positive co-prime integers, we set $\sq^{d/n}:=\sigma_{q^{d/n}} $.

\pagebreak[3]
\begin{lem}\label{lem1}
Let $\mu:=\frac{n}{d}\in \Q_{>0}$ with $n,d$ positive co-prime integers, and $[\l] \in \C^{*}/q^{\frac{d\Z}{n}}$. 

\begin{itemize}
\item 
Let $\hat{h}\in \C[[z]]$. Then, we have  $\hat{\mathcal{B}}_{\mu}\left( z\sq^{d/n}\left(\hat{h}\right)\right)= \z\hat{\mathcal{B}}_{\mu}\left( \hat{h}\right)$.
\item 
Let $f\in \mathbb{H}_{\mu}^{[\l]}$. 
Then, we have $z\sq^{d/n}\mathcal{L}_{\mu}^{[\l]}\left( f\right)= \mathcal{L}_{\mu}^{[\l]}\left( \z f\right)$.
\end{itemize}
\end{lem}

\begin{proof}
The first point is given by \cite{D4}, Lemma 1.4. The second point is a consequence of \cite{D4}, Lemma 1.4, and Remark \ref{rem3}.
\end{proof}

\begin{rem}\label{rem2}
Let $\hat{h}\in \C[[z]]$ and $f\in \mathbb{H}_{\mu}^{[\l]}$. Iterating $n$ times Lemma \ref{lem1}, we find $\hat{\mathcal{B}}_{\mu}\left( z^{n}\sq^{d}\left(\hat{h}\right)\right)= q^{\frac{-d(n-1)}{2}}\z^{n}\hat{\mathcal{B}}_{\mu}\left( \hat{h}\right)$ and $z^{n}\sq^{d}\mathcal{L}_{\mu}^{[\l]}\left( f\right)= q^{\frac{d (n-1)}{2}} \mathcal{L}_{\mu}^{[\l]}\left( \z^{n} f\right)$.

\end{rem}
Let $\C\{z\}$ be the ring of germs of analytic functions at $z=0$.
\pagebreak[3]
\begin{lem}\label{lem2}
Let $\mu:=\frac{n}{d}\in \Q_{>0}$ with $n,d$ positive co-prime integers, $[\l] \in \C^{*}/q^{\frac{d\Z}{n}}$. Let $f\in \C\{z\}$. Then,
\begin{itemize}
\item $\hat{\mathcal{B}}_{\mu}\left( f\right)\in \mathbb{H}_{\mu}^{[\l]}$.
\item $\mathcal{L}_{\mu}^{[\l]}\circ \hat{\mathcal{B}}_{\mu}\left( f\right)=f$.
\end{itemize}
\end{lem}

\begin{proof}
Let us prove briefly the first point. Let us take $f=\displaystyle\sum_{\ell\in\N}a_\ell z^\ell\in\C\{z\}$. It's easy to see that its $q$-Borel transform is defined and analytic on the whole plan $\C$, so we just have to check the $q^{\frac{d}{n}}$-exponential growth at infinity. Since $f\in \C\{z\}$, we have the existence of $L,M>0$ such that for all $\ell\in \N$, $|a_{\ell}|<LM^{\ell}$. We have  for all $z\in \C$,
\[\begin{array}{rcl}
 \left|\hat{\mathcal{B}}_{\mu}(f)(z)\right| & \leq & \displaystyle\sum_{\ell\in\N}|a_\ell||q|^{\frac{-\ell(\ell-1)d}{2n}}|z|^\ell\\
& \leq & L\displaystyle\sum_{\ell\in\N}M^{\ell}|q|^{\frac{-\ell(\ell-1)d}{2n}}|z|^\ell\\
& \leq &L \Theta_{|q|^{\frac{d}{n}}}(M|z|).
\end{array}\]

Let us now prove the second point. Using the expression of $\T_{q}$, we find that for all $k\in \Z$, $$\T_{q^{\frac{d}{n}}}\left(q^{k\frac{d}{n}}z\right)=q^{\frac{k(k-1)d}{2n}}z^{k}\T_{q^{\frac{d}{n}}}(z).$$ 
Following the definition of the $q$-Laplace transformation, we obtain that, for all ${[\l]\in \C^{*}/q^{\frac{d\Z}{n}}}$,  
$$\mathcal{L}_{\mu}^{[\l]}\left(1\right)=1.$$
We claim that for all $a\in \C$, $\ell\in \N$, $\mathcal{L}_{\mu}^{[\l]}\circ \hat{\mathcal{B}}_{\mu}\left( az^{\ell }\right)=az^{\ell }$.
Let us prove the claim by an induction on $\ell$. The case $\ell=0$ is a straightforward consequence of $\hat{\mathcal{B}}_{\mu}\left( 1\right)=1$, $\mathcal{L}_{\mu}^{[\l]}\left(1\right)=1$, and the $\C$-linearity of the two maps.\par 
Let $\ell\in \N$, and assume that for every $a\in \C$, we have  $\mathcal{L}_{\mu}^{[\l]}\circ \hat{\mathcal{B}}_{\mu}\left( az^{\ell }\right)=az^{\ell }$. Let $a\in \C^{*}$ and let us prove that $\mathcal{L}_{\mu}^{[\l]}\circ \hat{\mathcal{B}}_{\mu}\left( az^{\ell+1}\right)=az^{\ell+1 }$. We use Lemma \ref{lem1} to find 
$ 
\mathcal{L}_{\mu}^{[\l]}\circ \hat{\mathcal{B}}_{\mu}\left(  az^{\ell+1}\right)=a z \mathcal{L}_{\mu}^{[\l]}\circ \hat{\mathcal{B}}_{\mu}\left(  q^{ \frac{-\ell n}{d}}z^{\ell }\right).$ We use the induction hypothesis to obtain
$$az\sq^{d/n} \mathcal{L}_{\mu}^{[\l]}\circ \hat{\mathcal{B}}_{\mu}\left(q^{ \frac{-\ell n}{d}} z^{\ell }\right)=az\sq^{d/n} \left(q^{ \frac{-\ell n}{d}} z^{\ell }\right)=az^{\ell+1}.$$
This proves the claim. To prove the lemma, we remark that $\hat{\mathcal{B}}_{\mu}$ (resp. $\mathcal{L}_{\mu}^{[\l]}$) is an additive morphism between $\C\{z\}$ and $\mathbb{H}_{\mu}^{[\l]}$ (resp. $\mathbb{H}_{\mu}^{[\l]}$ and $\mathcal{M}(\C^{*},0)$).
\end{proof}

We finish this section by giving asymptotic properties of the Borel-Laplace summation. The definition of $q$-Gevrey asymptotic is an adaptation of \cite{D4}, Definition 1.8, in this setting. 

\pagebreak[3]
\begin{defi}
Let $\mu:=\frac{n}{d}\in \Q_{>0}$ with $n,d$ positive co-prime integers, $[\l] \in \C^{*}/q^{\frac{d\Z}{n}}$, $f\in\mathcal{M}(\C^{*},0)$ and ${\hat{f}:=\displaystyle\sum_{i\in \N}\hat{f}_{i}z^{i}\in \C[[z]]}$. We say that $$f\displaystyle\sim^{[\l]}_{\mu}\hat{f},$$ if for all $\e,R>0$ sufficiently small, there exist $L,M>0$, such that for all $k\in \N$ and for all 
$z$ in $$\Big\{z\in \C^{*}\Big| |z|<R\Big\}\setminus \displaystyle\bigcup_{\ell\in \frac{d\Z}{n}}\Big\{z\in \C^{*}\Big| \left|z+ q^{\ell}\l\right|<\e\left| q^{\ell}\l \right|\Big\},$$
we have 
$$\Bigg|f(z)-\displaystyle\sum_{i=0}^{k-1}\hat{f}_{i}z^{i} \Bigg|<LM^{k}|q|^{\frac{k(k-1)}{2\mu}}|z|^{k}.$$
\end{defi}
For $\mu\in \Q_{>0}$, we define the formal $q$-Laplace transformation of order $\mu$ as follows:
$$
\begin{array}{llll}
\hat{\mathcal{L}}_{\mu}:&\C[[\z]]&\longrightarrow&\C[[z]]\\
&\displaystyle\sum_{\ell\in \N} a_{\ell}\z^{\ell}&\longmapsto&\displaystyle\sum_{\ell\in \N} a_{\ell}q^{\frac{\ell(\ell-1)}{2\mu}} z^{\ell}.
\end{array}
$$
Using Lemma \ref{lem1}, for all $\ell\in \N$, and  $[\lambda]\in \C^{*}/q^{\frac{d\Z}{n}}$, we obtain $$\widehat{\mathcal{L}}_{\mu}\left(\z^{\ell}\right)=\mathcal{L}_{\mu}^{[\l]}\left(\z^{\ell}\right).$$
\pagebreak[3]
\begin{propo}
Let $\mu:=\frac{n}{d}\in \Q_{>0}$ with $n,d$ positive co-prime integers, $[\l] \in \C^{*}/q^{\frac{d\Z}{n}}$ and  $f\in  \mathbb{H}^{[\l]}_{\mu}$. Then, we have
$$\mathcal{L}_{\mu}^{[\l]}\left( f\right)\displaystyle\sim^{[\l]}_{\mu}\widehat{\mathcal{L}}_{\mu}\left(f\right). $$
\end{propo}

\begin{proof}
This is a direct consequence of \cite{D4}, Proposition 1.9, and Remark \ref{rem3}.
\end{proof}

\begin{coro}
Let $\mu:=\frac{n}{d}\in \Q_{>0}$ with $n,d$ positive co-prime integers, $[\l] \in \C^{*}/q^{\frac{d\Z}{n}}$ and $f\in \C[[z]]$ with  $\hat{\mathcal{B}}_{\mu}\left( f\right)\in \mathbb{H}_{\mu}^{[\l]}$. Then $\mathcal{L}_{\mu}^{[\l]}\circ \hat{\mathcal{B}}_{\mu}\left( f\right)\displaystyle\sim^{[\l]}_{\mu} f$.
\end{coro}
 \pagebreak[3]
\section{Local analytic study of $q$-difference systems.}\label{sec2}
 
Let $\C((z))$ be the fraction field of  $\C[[z]]$. Let $K$ be a sub-field of $\C((z))$ or $\mathcal{M}(\C^{*},0)$ stable under $\sq$ and let ${A,B\in \mathrm{GL}_{m}(K)}$. The two $q$-difference systems, $\sq Y=AY$ and $\sq Y=BY$ are said to be equivalent over $K$, if there exists $P\in \mathrm{GL}_{m}(K)$, called the gauge transformation, such that $$B=P[A]_{\displaystyle\sq}:=(\sq P)AP^{-1}.$$ In particular, $$\sq Y=AY\Longleftrightarrow\sq \left(PY\right)=BPY.$$ 
Conversely, if there exist $A,B,P\in \mathrm{GL}_{m}(K)$
such that
$\sq Y=AY$, $ \sq Z=BZ$ and $Z=PY$, then 
$$B=P\left[ A\right]_{\displaystyle\sq}.$$\par 
We remind that $\C(\{z\})$ is the field of germs of meromorphic functions at $z=0$. Until the end of the paper, we fix ${A\in \mathrm{GL}_{m}\Big(\C(\{z\})\Big)}$. 
Let $\mu:=\frac{n}{d}\in \Q_{>0}$ with $n,d$ positive co-prime integers, and $a\in \C^{*}$. We define the $d$ times $d$ companion matrix as follows:
$$E_{n,d,a}:=\begin{pmatrix}
0&1&&\\
&\ddots&\ddots& \\
&&\ddots&1 \\
az^{n}&&&0
\end{pmatrix}.$$
For the proof of the following result, see \cite{BuT}, Theorem 1.18. See also \cite{VdPR}, Corollary 1.6.
\begin{theo}\label{theo2}
There exist integers ${n_{1},\dots,n_{k}\in \Z}$, ${d_{1},\dots,d_{k},m_{1},\dots,m_{k}\in \N^{*}}$, with $\frac{n_{1}}{d_{1}}<\dots<\frac{n_{k}}{d_{k}}$,~$\gcd(n_{i},d_{i})=1$, ${\sum m_{i}d_{i}=m}$, and 
\begin{itemize}
\item $\hat{H}\in\mathrm{GL}_{m}\Big(\C((z))\Big)$,
\item $C_{i}\in \mathrm{GL}_{m_{i}}(\C)$,
\item $a_{i}\in \C^{*}$,
\end{itemize}
 such that $A=\hat{H}\left[\mathrm{Diag}\Big(C_{i}\otimes E_{n_{i},d_{i},a_{i}}\Big)\right]_{\displaystyle\sq}$, where $$\mathrm{Diag}\Big(C_{i}\otimes E_{n_{i},d_{i},a_{i}}\Big):=\begin{pmatrix}
C_{1}\otimes E_{n_{1},d_{1},a_{1}} &&\\
&\ddots&\\
&&C_{k}\otimes E_{n_{k},d_{k},a_{k}} 
\end{pmatrix}.$$
\end{theo}
The aim of this paper will be to replace this $\hat{H}$ by  a meromorphic gauge transformation computed with $q$-Borel-Laplace transformations.\\
 Assume that $k=2$, see Theorem \ref{theo2} for the notation, and let $n_{1},n_{2},d_{1},d_{2},a_{1},a_{2},C_{1},C_{2}$, given by Theorem \ref{theo2}. In \cite{BuT}, V. Bugeaud makes the local analytic classification of $q$-difference systems with two slopes. Let us now detail her work. She explains how to reduce to the case where ${C_{2}\otimes E_{n_{2},d_{2},a_{2}}=\mathrm {I}_{1}}$, that is the identity matrix of size $1$ and where $C_1$ is a unipotent Jordan bloc. For $r\geq 1$ let us write 
 \[U_r:=\begin{pmatrix}
 1 & 1 & & 0\\
 & \ddots & \ddots & \\
 & & \ddots & 1\\
 0 & & & 1
 \end{pmatrix}\in \mathrm{GL}_{r}(\C),\]
the unipotent Jordan bloc of length $r$.

\begin{center}\textbf{Until the end of the paper, we assume that $k=2$, $C_{2}\otimes E_{n_{2},d_{2},a_{2}}=\mathrm {I}_{1}$ and $C_1=U_r$, for some $r\geq 1$.}\end{center}
Note that due to the assumption, we have $\frac{n_{2}}{d_{2}}=0$ and $\frac{n_{1}}{d_{1}}<0$. Let $n:=-n_{1}\in \N^{*},d:=d_{1}$, $\mu:=\frac{n}{d}\in \Q_{>0}$ and $a:=a_{1}$. Note that $m-1=d\times r$. Let us remind that $\C\{z\}$ is the ring of germs of analytic functions at $z=0$. Using  \cite{BuT}, Theorem 2.9, we deduce:
\pagebreak[3]
\begin{theo}\label{theo1}
There exist
\begin{itemize}
\item $W$, column vector of size $dr$ with coefficients in $\displaystyle \sum_{\nu=0}^{n-1}\C z^{\nu}$,
\item $F\in\mathrm{GL}_{m}\Big(\C(\{z\})\Big)$, 
\end{itemize}
such that $A=F[B]_{\displaystyle\sq}$, where:
$$B:=\begin{pmatrix}
E_{-n,d,a}\otimes U_r&W\\
0&1
\end{pmatrix}. $$
\end{theo}

It follows from Theorem \ref{theo2} that we have the existence and the uniqueness of $$\begin{pmatrix}
\mathrm {I}_{dr}&\hat{H} \\
0&1
\end{pmatrix}\in \mathrm{GL}_{m}\Big(\C[[z]]\Big),$$ formal gauge transformation, that satisfies
$$B=\begin{pmatrix}
\mathrm {I}_{dr}&\hat{H} \\
0&1
\end{pmatrix}\Big[\mathrm{Diag} (E_{-n,d,a}\otimes U_r,1)\Big]_{\displaystyle\sq}.$$ 
 Let us define the finite set $\Sigma\subset \C^{*}/q^{\frac{d\Z}{n}}$, as follows
$$\Sigma := \left\{ [\l]\in \C^*/q^{\frac{d\Z}{n}} \Big| \lambda^n \in aq^{\frac{(n-1)d}{2}} q^{d\Z}\right\}.$$
 Let $\mathcal{M}(\C^{*})$ be the field of meromorphic functions on $\C^{*}$.

\pagebreak[3]
\begin{propo}\label{propo1}
Let $[\lambda]\in \left(\C^{*}/q^{\frac{d\Z}{n}}\right)\setminus \Sigma$. There exists a unique matrix $$\begin{pmatrix}
\mathrm {I}_{dr}&\hat{H}^{\left[\l\right]} \\
0&1
\end{pmatrix}\in \mathrm{GL}_{m}\Big(\mathcal{M}(\C^{*})\Big),$$ such that 
$B=\begin{pmatrix}
\mathrm {I}_{dr}&\hat{H}^{\left[\l\right]} \\
0&1
\end{pmatrix}\Big[\mathrm{Diag} ( E_{-n,d,a}\otimes U_r,1)\Big]_{\displaystyle\sq}$, and  for all $1\leq  i \leq dr$, the meromorphic function of the line number $i$ of the column vector $\hat{H}^{\left[\l\right]}$, has poles of order at most~$1$ contained in the $q^{\frac{d\Z}{n}}$-spiral $-\l q^{-i+1}q^{\frac{d\Z}{n}}$.  
\end{propo}

\begin{rem}
In \cite{BuT}, Proposition 3.8, a similar result is shown. In her work, the entry number $i$ of the meromorphic gauge transformation  has poles of order at most $n$ contained in the $q^{d\Z}$-spiral $-\l q^{-i+1}q^{d\Z}$. It is worth mentioning that our meromorphic gauge transformation has the same number of poles, counted with multiplicities, which is a crucial property.
\end{rem}

\begin{proof}
Using $B=\begin{pmatrix}
\mathrm {I}_{dr}&\hat{H}\\
0&1
\end{pmatrix}\Big[\mathrm{Diag} (E_{-n,d,a}\otimes U_r,1 )\Big]_{\displaystyle\sq}$, we find,

\begin{equation}\label{eq5}
\sq \left(\hat{H}\right)=\left(E_{-n,d,a}\otimes U_r\right)\hat{H}+W.
\end{equation}
Let us write $(\hat{h}_{i})_{i=1,\dots, dr}:=\hat{H}$ and $(w_{i})_{i=1,\dots, dr}:=W$. Writing the equations of \eqref{eq5}, we find for all $1\leq i<d$, for all $0\leq k< r$, (we make the convention that $\hat{h}_{i}=0$ for $i>dr$)

\begin{equation}\label{eq3}
\sq \hat{h}_{i+kd}=\hat{h}_{i+1+kd}+\hat{h}_{i+1+(k+1)d}+w_{i+kd}.
\end{equation}
And for all $0< k\leq r$,

\begin{equation}\label{eq4}
z^{n}\sq \hat{h}_{kd}=a\hat{h}_{(k-1)d+1}+a\hat{h}_{kd+1}+w_{kd}.
\end{equation}

For $1\leq i\leq d$, let $\hat{H}_{i}:=(\hat{h}_{i+kd})_{k=0,\dots, r-1}$. With \eqref{eq3} and \eqref{eq4}, we find that there exist  $W_{1},\dots, W_{d}\in (\C[z])^{r}$ such that  \eqref{eq5} is equivalent to 
$$
z^{n}\sq^{d}\hat{H}_{1}=aU_{r}^d\hat{H}_{1}+W_{1}, \hbox{ and for }1<i \leq d, \; \sq \hat{H}_{i-1}=U_r\hat{H}_{i}+W_{i}.
$$

Consequently, it is sufficient to prove the existence and the uniqueness of a meromorphic vector $\hat{H}_{1}^{[\l]}\in \left(\mathcal{M}(\C^{*})\right)^{r}$ with entries having poles of order at most~$1$ contained in the $q^{\frac{d\Z}{n}}$-spiral $-\l q^{\frac{d\Z}{n}}$, satisfying
$$z^{n}\sq^{d} \hat{H}_{1}^{[\l]}=a U_r^d \hat{H}_{1}^{[\l]} +W_{1}. $$
If we put $q':=q^{d}$, and $z':=z^{n}$, we find that the the case $W_{1}\in \C[z^{n}]$ is an application of \cite{RSZ}, Proposition 3.3.1, combined with  \cite{RSZ}, Theorem 6.1.2. Let us write ${W_{1}=\sum_{\kappa=0}^{d-1}W_{1,\kappa}z^{\kappa}}$ with $W_{1,\kappa}\in \C[z^{n}]$ and consider the unique meromorphic solution $\hat{H}_{1,\kappa}^{[\l]}$ of  $z^{n}\sq^{d} \hat{H}_{1,\kappa}^{[\l]}=a U_r^d \hat{H}_{1,\kappa}^{[\l]} +W_{1,\kappa}$ with convenient poles. Then $\hat{H}_{1}^{[\l]}:= \sum_{\kappa=0}^{d-1}\hat{H}_{1,\kappa}^{[\l]}z^{\kappa}$ is a meromorphic solution of $z^{n}\sq^{d} \hat{H}_{1}^{[\l]}=a U_r^d \hat{H}_{1}^{[\l]} +W_{1}$ with convenient poles. Let us prove the uniqueness. To the contrary, let us assume that there exists another meromorphic solution with the same properties for the poles. Then the equation $z^{n}\sq^{d} Y=a U_r^d Y$ has a non zero meromorphic solution having poles of order at most~$1$ contained in the $q^{\frac{d\Z}{n}}$-spiral $-\l q^{\frac{d\Z}{n}}$. This contradict the uniqueness in the case $W_{1}\in \C[z^{n}]$, and proves the result.
\end{proof}

\pagebreak[3]
\section{Main result}\label{sec3}

Let us keep the notations of the previous section. We now state our main result. We refer to $\S \ref{sec1}$ for the definitions of the Borel and Laplace transformations. Let ${[\l]\in \left(\C^{*}/q^{\frac{d\Z}{n}}\right)\setminus \Sigma}$.
We have seen in the proof of Proposition \ref{propo1}, that in order to compute  $\hat{H}^{\left[\l\right]}=:\left(\hat{h}^{\left[\l\right]}_{i}\right)_{i=1,\dots,dr}$ it is sufficient to compute $\hat{H}^{\left[\l\right]}_{1}:=\left(\hat{h}^{\left[\l\right]}_{1+kd}\right)_{k=0,\dots,r-1}$.
Let us write $\hat{H}=:\left(\hat{h}_{i}\right)_{i=1,\dots,dr}$, and  $\hat{H}_{1}:=\left(\hat{h}_{1+kd}\right)_{k=0,\dots,r-1}$.

\pagebreak[3]
\begin{theo}\label{theo3}
We have  $\mathcal{\hat{B}}_{\mu}\left(\hat{H}_{1}\right)\in \left(\mathbb{H}_{\mu}^{[\l]}\right)^{r}$ and 
$$\hat{H}_{1}^{\left[\l\right]}= \mathcal{L}_{\mu}^{[\l]}\circ\hat{\mathcal{B}}_{\mu}\left(\hat{H}_{1}\right).$$
\end{theo}

\begin{proof}
We remind that we have $z^{n}\sq^{d} \hat{H}_{1}^{\left[\l\right]}=aU_{r}^{d}\hat{H}_{1}^{\left[\l\right]}+W_{1}$ for some $W_{1}\in (\C[z])^{r}$. \par 
 We claim that  $\hat{\mathcal{B}}_{\mu}\left(\hat{H}_{1}\right)\in \left(\mathbb{H}_{\mu}^{[\l]}\right)^{r}$. 
We use Remark  \ref{rem2}, to find that

\begin{equation}\label{eq2}
q^{\frac{-d (n-1)}{2}}\z^{n}\hat{\mathcal{B}}_{\mu} \left(\hat{H}_{1}\right)=aU_{r}^d\hat{\mathcal{B}}_{\mu}\left(\hat{H}_{1}\right)+\hat{\mathcal{B}}_{\mu}(W_{1}).
\end{equation}
Using the definition of $\Sigma$, we deduce that for all $\k \in \Z$,

$$\l^{n}q^{\frac{-d (n-1)}{2}} q^{\k d}\neq a.$$
 This proves our claim. \par 
With Lemma \ref{lem2}, we obtain that $\mathcal{L}_{\mu}^{[\l]}\circ\hat{\mathcal{B}}_{\mu}\left( W_{1}\right)=W_{1}$.
Using additionally (\ref{eq2}) and Remark \ref{rem2},  we find that
$$z^{n}\sq^{d} \left(\mathcal{L}_{\mu}^{[\l]}\circ\hat{\mathcal{B}}_{\mu}\left(\hat{H}_{1} \right)\right)=aU_{r}^d\mathcal{L}_{\mu}^{[\l]}\circ\hat{\mathcal{B}}_{\mu}\left(\hat{H}_{1} \right)+W_{1}.$$
Moreover, the entries of $\mathcal{L}_{\mu}^{[\l]}\circ\hat{\mathcal{B}}_{\mu}\left(\hat{H}_{1}\right)$  have poles of order at most~$1$ contained in the $q^{\frac{d\Z}{n}}$-spiral $-\l q^{\frac{d\Z}{n}}$.
But we have seen in the proof of  Proposition \ref{propo1}, that this implies $$\hat{H}_1^{[\lambda]} = \mathcal{L}_{\mu}^{[\l]}\circ\hat{\mathcal{B}}_{\mu}\left(\hat{H}_{1} \right).$$
\end{proof}

\begin{ex}
Assume that $n=1$, $d=2$, $r=1$, $a=1$, and $W=(0,1)^{t}$. Then, the first entry $\hat{h}$ of $\hat{H}$ is solution of 
$$z\sq^{2} \left(\hat{h}\right)=\hat{h}+ z.$$
 We have 
$$\z\hat{\mathcal{B}}_{1/2} \left(\hat{h}\right)=\hat{\mathcal{B}}_{1/2}\left(\hat{h}\right)+ \z,$$
which means that $\hat{\mathcal{B}}_{1/2}\left(\hat{h}\right)=\frac{\z}{\z-1}$. Therefore, for all $[\l]\in \left(\C^{*}/q^{2\Z}\right)\setminus \{[1]\}$ , we find that 
$\hat{H}^{\left[\l\right]}=(\hat{h}^{[\l]},\sq \hat{h}^{[\l]})^{t}$, where 
$\hat{h}^{[\l]}$
$$\hat{h}^{[\l]}=\displaystyle \sum_{\ell\in 2\Z}\frac{q^{\ell}\l}{q^{\ell}\l-1}\times\frac{1}{\T_{q^{2}}\left(\frac{q^{2+\ell}\l}{z}\right)}.$$
\end{ex}

\pagebreak[3]
\begin{rem}\label{rem1}
In \cite{D4}, Theorem 2.5, the author defines a meromorphic solution with another $q$-analogues of the Borel and Laplace transformation. When $d=1$, the two sums obtained are the same. On the other hand when $d\neq 1$ the solutions defined here have less poles than the solutions of \cite{D4}. More precisely, see Theorem \ref{theo2} for the notations, assume that $k=2$ and $n_{2}=0$. Let $d_{1},n_{1}$ be given by Theorem \ref{theo2}. In \cite{D4}, it is shown how to build, with another $q$-Borel and $q$-Laplace transformations, meromorphic gauge transformations similar to $\S \ref{sec2}$ but with germs of entries having $q^{\frac{\Z}{|n_{1}|}}$-spirals of poles of order $1$. The solutions we compute here  have  $q^{\frac{d_{1}\Z}{|n_{1}|}}$-spirals of poles of order $1$. 
\end{rem}

\fussy
\pagebreak[3]

\bibliographystyle{alpha}
\bibliography{biblio}

\end{document}